\documentclass[reqno]{amsart}
\usepackage{mathrsfs}
\usepackage{hyperref}
\usepackage{float}
\usepackage{graphicx}

\usepackage{enumerate}
\newtheorem{theorem}{Theorem}[section]
\newtheorem{lemma}[theorem]{Lemma}

\theoremstyle{definition}
\newtheorem{definition}[theorem]{Definition}

\newtheorem{remark}[theorem]{Remark}

\newcommand{\norm}[1]{\left\Vert#1\right\Vert}

\numberwithin{equation}{section}
\usepackage{amsfonts}

\usepackage{amsmath}
\usepackage{amssymb}
\usepackage{cite}
\usepackage{hyperref}
\hypersetup{
   colorlinks,
    citecolor=blue,
    filecolor=black,
    linkcolor=blue,
    urlcolor=magenta
}

\begin{document}
\font\nho=cmr10
\def\dive{\mathrm{div}}
\def\cal{\mathcal}
\def\L{\cal L}

\def \ud{\underline }
\def\id{{\indent }}
\def\f{\frac}
\def\non{{\noindent}}
 \def\le{\leqslant} 
 \def\leq{\leqslant}
 \def\geq{\geqslant} 
\def\rar{\rightarrow}
\def\Rar{\Rightarrow}
\def\ti{\times}
\def\i{\mathbb I}
\def\j{\mathbb J}
\def\si{\sigma}
\def\Ga{\Gamma}
\def\ga{\gamma}
\def\ld{{\lambda}}
\def\Si{\Psi}
\def\f{\mathbf F}
\def\r{\hro{R}}
\def\e{\cal{E}}
\def\B{\cal B}
\def\A{\mathcal{A}}
\def\p{\mathbb P}

\def\tet{\theta}
\def\Tet{\Theta}
\def\hro{\mathbb}
\def\ho{\mathcal}
\def\P{\ho P}
\def\E{\mathcal{E}}
\def\n{\mathbb{N}}
\def\M{\mathbb{M}}
\def\dMu{\mathbf{U}}
\def\dMcs{\mathbf{C}}
\def\dMcu{\mathbf{C^u}}
\def\vk{\vskip 0.2cm}
\def\td{\Leftrightarrow}
\def\df{\frac}
\def\Wei{\mathrm{We}}
\def\Rey{\mathrm{Re}}
\def\s{\mathbb S}
\def\l{\mathcal{L}}
\def\C+{C_+([t_0,\infty))}
\def\o{\cal O}

\title[AAP-Solutions of INSE]{On Asymptotically Almost Periodic Solutions of the parabolic-elliptic Keller-Segel system on real hyperbolic Manifolds}

\author[T.V. Thuy]{Tran Van Thuy}
\address{Tran Van Thuy\hfill\break
East Asia University of Technology, Trinh Van Bo, Nam Tu Liem, Hanoi, Vietnam}
\email{thuyhum@gmail.com or thuytv@eaut.edu.vn}

\author[N.T. Van]{Nguyen Thi Van}
\address{Nguyen Thi Van \hfill\break
Faculty of Computer Science and Engineering, Thuyloi University, 175 Tay Son, Dong Da, Hanoi,
Vietnam}
\email{van@tlu.edu.vn}

\author[P.T. Xuan]{Pham Truong Xuan}
\address{Pham Truong Xuan \hfill\break
Thang Long Institute of Mathematics and Applied Sciences (TIMAS), Thang Long University,
Nghiem Xuan Yem, Hoang Mai, Hanoi, Vietnam.} 
\email{phamtruongxuan.k5@gmail.com or xuanpt@thanglong.edu.vn} 

\begin{abstract}  
In this article, we investigate the existence, uniqueness and exponential decay of asymptotically almost periodic solutions of the parabolic-elliptic Keller-Segel system on a real hyperbolic manifold. We  prove the existence and uniqueness of such solutions in the linear equation case by using the dispersive and smoothing estimates of the heat semigroup. Then we pass to the well-posedness of {the} semi-linear equation case by using the results of linear equation and fixed point arguments. The exponential decay is proven by using Gronwall's inequality. 
\end{abstract}



\maketitle


\section{Introduction}
We consider the  Keller–Segel model on the real hyperbolic space $\mathbb{H}^n \,\, (n \geqslant 2)$ described by the following equations:
\begin{equation}\label{KSH} 
\left\{
  \begin{array}{rll}
u_t \!\! &= \Delta_{\mathbb{H}^n}u - \chi\nabla_x \cdot (u\nabla_x v) + g(t) \quad  &x\in  \mathbb{H}^n,\, t \in \r_+, \hfill \cr
-\Delta_{\mathbb{H}^n} v + \gamma v \!\!&=\; \alpha u \quad &x\in \mathbb{H}^n,\,  t \in \r_+,\cr
u(0) &= u_0(x) \quad &x\in \mathbb{H}^n,
\end{array}\right.
\end{equation}
where the operator $\Delta_{\mathbb{H}^n}$ means Laplace-Beltrami operator on $\mathbb{H}^n$; $u(t,x): \mathbb{R}_+\times \mathbb{H}^n \to \mathbb{R}_+$ is the density of cells; $v(t,x): \mathbb{R}_+\times \mathbb{H}^n \to \mathbb{R}_+$ is the concentration of the chemoattractant; $\chi$ is the positive sensitivity parameter; $\gamma \geqslant 0$ and $\alpha >0$ represent the decay and production rate of the chemoattractant, respectively. Assume that the given function $g(t)$ is asymptotically almost periodic.

The original model  was contributed by Keller and Segel (see \cite{KeSe}). The ideal for the model {is} based on biological phenomena describing chemotaxis (aggregation of organisms sensitive to a gradient of a chemical substance). For forty years, because of significant applications in biology, various versions of the model have been widely studied.  In the Euclidean space $\mathbb{R}^n$ when $ n=2$ some crucial results are obtained by A.Blanchet, J. Dolbeault, L. Corrias,  (see \cite{Bla, Do, CoEs}) and when $ n \geq 3$, a large series of results are from the deep works of Chen, Ferreira, Hieber, Iwabuchi, Konozo, Lee, Sugiyama, Yahagi,  Wang (see  \cite{Li2023, Qi2017, Chen2018, Ko1, Ko2, Fe2021, Hi2020, Iwa2011}). Furthermore, on the hyperbolic manifold  $\mathbb{H}^2$, Pierfelice and Maheux  recently proved the local and global well-posedness results under the sub-critical condition in \cite{MaPi}. Xuan et al. (in \cite{Xu}) continued to prove  the existence and unique  periodic mild solution of the Keller-Segel system on the whole spaces,  both  the Euclidean space $\mathbb{R}^n , n\geq 4$   and  the hyperbolic space $ \mathbb{H}^n, n \geq 2$.

One should also recall that the concept of almost periodicity  was firstly introduced by Bohr in the mid-twenties (see \cite{Bo1, Bo2, Bo3}). Afterwards, the theory of almost periodic functions was continuously developed  by some mathematicians like Amerio and Prouse \cite{Ame}, Levitan \cite{Le}, Besicovitch \cite{Be}, Bochner\cite{Boc}, Neumann, Fr\'echet, Pontryagin, Lusternik, Stepanov, Weyl, etc. (see \cite{Che, Dia}). The concept of asymptotically almost periodicity was introduced in the literature by the French mathematician Fr\'echet. The works of known authors dedicated to asymptotically almost periodicity of solutions of differential equations that were introduced in the book \cite{Dia}. Concerning the asymptotically almost periodic solution of equations in fluid dynamic systems, in recent works \cite{XVQ,XV}, Xuan and his collaborations have answered fully the problems on the existence, uniqueness and exponential stability of such type solutions for Navier-Stokes equations in $L^p(\mathbb H^n)$ spaces for all $1<p$. The method is based on dispersive estimates of vectorial heat equations, {the} Massera-type principle and fixed point arguments. These works \cite{XVQ, XV}, in fact, provided a positive answer for the existence of asymptotically almost periodic mild solution of Navier-Stokes equations in the framework of manifolds with negative Ricci curvature, meanwhile this existence is still an open question for Navier-Stokes equations in Euclidean spaces and its domains. In this paper, we extend \cite{XVQ, XV} to study the asymptotically almost periodic mild solutions for the above Keller-Segel system \eqref{KSH}.

In particular, we first consider inhomogeneous linear system corresponding to system \eqref{KSH}:
\begin{equation}\label{KSH0} 
\left\{
  \begin{array}{rll}
u_t \!\! &= \Delta_{\mathbb{H}^n}u - \chi\nabla_x \cdot (\omega\nabla_x v) + g(t) \quad  &x\in  \mathbb{H}^n,\, t \in \r_+, \hfill \cr
-\Delta_{\mathbb{H}^n} v + \gamma v \!\!&=\; \alpha \omega \quad &x\in \mathbb{H}^n,\,  t \in \r_+,\cr
u(0) &= u_0(x) \quad &x\in \mathbb{H}^n,
\end{array}\right.
\end{equation}
for a given function $\omega \in C_b(\mathbb{R}_+, L^{\frac{p}{2}}(\mathbb{H}^n))$ (where ${\max \left\lbrace 3,n\right\rbrace }<p<2n$). We establish that the system \eqref{KSH0} has a unique mild solution in $C_b(\mathbb{R}_+, L^{\frac{p}{2}}(\mathbb{H}^n))$. Moreover, we also prove that the system \eqref{KSH0} on the whole line time-axis (without initial data $u_0$) has a unique mild solution in $C_b(\r, L^{\frac{p}{2}}(\mathbb{R}^n))$ for a given $\omega \in C_b(\mathbb{R}, L^{\frac{p}{2}}(\mathbb{H}^n))$ (see Lemma \ref{Thm:linear}). Hence, we obtain a Massera-type principle that if $g$ and $\omega$ are asymptotically almost periodic functions in suitable spaces then the system \eqref{KSH0} has a unique asymptotically almost periodic mild solution (see Theorem \ref{pest}). Combining the results of linear Keller-Segel system and fixed point arguments, we establish the well-posedness of asymptotically almost periodic mild solutions for semi-linear system \eqref{KSH} in the small ball of $C_b(\r_+,L^{\frac{p}{2}}({\mathbb{H}}^n))$.
In addition, the exponential decay of such solutions will be obtained by using Gronwall's inequality (see Theorem \ref{thm2.20}).

Our paper is organized in three sections. In Section \ref{S2}, we give some preliminaries about asymptotically almost periodic ones to \eqref{KSH} in the real hyperbolic space $\mathbb{H}^n$. In Section \ref{S3}, we deal with the existence and uniqueness of almost periodicity and asymptotically almost periodicity to the linear system { corresponding} to  \eqref{KSH}.  Then, the paper ends with Section \ref{S4} in proving the existence, uniqueness and exponential decay  of asymptotically almost periodic mild solutions to the semi-linear  system \eqref{KSH}. \\

\section{Preliminaries}\label{S2}
\subsection{ The concepts of functions}
\bigskip
Let $X$ be a Banach space, we denote 
$$C_b(\r, X):=\{f:\r \to X \mid f\hbox{ is continuous on $\r$ and }\sup_{t\in\r}\|f(t)\|_X<\infty\}$$
which is a Banach space endowed with the norm $\|f\|_{\infty, X}=\|f\|_{C_b(\r, X)}:=\sup\limits_{t\in\r}\|f(t)\|_X$.
Similarly, we denote 
$$C_b(\r_+, X):=\{f:\r_+ \to X \mid f\hbox{ is continuous on ${\r_+}$ and }\sup_{t\in\r_+}\|f(t)\|_X<\infty\}$$
which is also a  Banach space endowed with the norm $\|f\|_{\infty, X}=\|f\|_{C_b(\r_+, X)}:=\sup\limits_{t\in\r_+}\|f(t)\|_X$.

\begin{definition}(Bohr \cite{Bo1,Bo2,Bo3})
A function  $h \in C_b(\r, X )$ is called almost periodic function if for each $ \epsilon  > 0$, there exists $l_{\epsilon}>0 $ such that every interval of length $l_{\epsilon}$ contains at least a number $T $ with the following property
\begin{equation}
 \sup_{t \in \r } \| h(t+T)  - h(t) \| < \epsilon.
\end{equation}
The collection of all almost periodic functions $h:\r \to X $ will be denoted by $AP(\r,X)$ which is a Banach space endowed with the norm $\|h\|_{ AP(\r,X)}=\sup_{t\in\r}\|h(t)\|_X.$
\end{definition}
To introduce the asymptotically almost periodic functions, we need the space  $C_0 (\r_+,X)$, that is, the collection of all continuous functions $\varphi: \r_+ \to X$ such that
 $$\mathop{\lim}\limits_{t \to \infty } \| \varphi(t) \|=0.$$
Clearly, $C_0 (\r_+,X)$  is a Banach space endowed with the norm $\|\varphi\|_{C_0 (\r_+,X)}=\sup_{t\in\r_+}\|\varphi(t)\|_X$.
\begin{definition} 
A continuous function  $f \in C(\r_+, X )$  is said to be asymptotically almost periodic if there exist  $h \in AP(\r,X)$ and $ \varphi\in C_0(\r_+,X)$ such that
\begin{equation}
f(t) = h(t) + \varphi(t).
\end{equation}
We denote $AAP(\r_+, X):= \{f:\r_+ \to X \mid f\hbox{ is asymptotically almost periodic on $\r_+$}\}$. Under the norm $\|f\|_{ AAP(\r_+,X)}=\|h\|_{ AP(\r, X)}+\|\varphi\|_{ C_0(\r_+,X)}$, then $AAP(\r_+,X)$ is a Banach space.
\end{definition} 
\begin{remark}
This definition is { arised from} the origin one of Fr\'echet: A function $f\in C_b( [0,\infty),X)$ is said to be asymptotically almost periodic if and only if,
for every $\epsilon>0$, we can find numbers $l > 0$ and $M > 0$ such that every subinterval of $[0,\infty)$ of length $l$ contains, at least, one number $\tau$ such that $|f(t + \tau) - f (t)|\leq \epsilon$ for all $t\geq M$.
\end{remark}
The decomposition of asymptotically almost periodic functions is unique \rm( see \cite[Proposition 3.44, page 97]{Dia}), that is, 
$$AAP(\r_+,X) = AP(\r,X) \oplus C_0(\r_+,X).$$
The relation of extension between the above Banach spaces can be showed  as
$$P(\r,X) \longrightarrow AP(\r,X) \longrightarrow AAP(\r_+,X) \longrightarrow C_b(\r,X).$$
where  $P(\r,X)$ is the space of all  periodic functions from  $\r$ to $X$.\\
{\bf Example.}
 The function $h(t)=\sin{ t}+\sin({\sqrt{2}t})$ is almost periodic but not periodic and also  $f(t) =\sin{ t}+\sin({\sqrt{2}t})+e^{-t}$  is asymptotically almost periodic but not almost periodic.

\subsection{The real hyperbolic manifold} 
 Let $(\mathbb{H}^{n},{\mathfrak g})= (\mathbb{H}^{n}(\mathbb{R}),{\mathfrak g})$ stand for a real hyperbolic manifold, where $n\geqslant 2$ is the dimension, endowed with a Riemannian metric ${\mathfrak g}$. This space is realized via a  hyperboloid in $\mathbb{R}^{n+1}$ by considering the upper sheet
$$
\left\{  (x_{0},x_{1},...,x_{n})\in\mathbb{R}^{n+1};\text{ }\,x_{0}\geq1\text{
and }x_{0}^{2}-x_{1}^{2}-x_{2}^{2}...-x_{n}^{2}=1\,\right\},$$
where the metric is given by $d{\mathfrak g} =-dx_{0}^{2}+dx_{1}^{2}+...+dx_{n}^{2}.$

In geodesic polar coordinates, the hyperbolic manifold $(\mathbb{H}^{n},{\mathfrak g})$ can be described as
$$
\mathbb{H}^{n}=\left\{  (\cosh\tau,\omega\sinh\tau),\,\tau\geq0,\omega
\in\mathbb{S}^{n-1}\right\}
$$
with $d{\mathfrak g}=d\tau^{2} +(\sinh\tau)^{2}d\omega^{2},$ where $d\omega^{2}$ is the
canonical metric on the sphere $\mathbb{S}^{n-1}$. In these coordinates, the
Laplace-Beltrami operator $\Delta_{\mathbb{H}^n}$ on $\mathbb{H}^{n}$ can be
expressed as
$$
\Delta_{\mathbb{H}^n}=\partial_{r}^{2}+(n-1)\coth r\partial
_{r}+\sinh^{-2}r\Delta_{\mathbb{S}^{n-1}}.
$$
It is well known that the spectrum of $-\Delta_{\mathbb{H}^n}$ is the half-line
$\left[\dfrac{(n-1)^2}{4},\infty \right)$.

\subsection{The Keller-Segel system on real hyperbolic manifolds} \label{2_3}
We consider the Keller-Segel system on $(\mathbb{H}^n,{\mathfrak g})$, where $n \geqslant 2$ 
\begin{equation}\label{KS} 
\left\{
  \begin{array}{rll}
u_t \!\! &= {\Delta_{\mathbb{H}^n}} u - \chi\nabla_x \cdot (u\nabla_x v) + g(t) \quad  &x\in  \mathbb{H}^n, t \in \r,\hfill \cr
- {\Delta_{\mathbb{H}^n}} v + \gamma v \!\!&=\; \alpha u \quad &x\in \mathbb{H}^n, t \in \r.
\end{array}\right.
\end{equation}
For the sake of convenience, we assume that the constant $\chi=1$. Moreover, for the simplicity of calculations we study the case $g(t)= \nabla_x \cdot f(t,\cdot)$.

The second equation of this system leads to $v= \alpha (- {\Delta_{\mathbb{H}^n}}+ \gamma I)^{-1}u$.
Therefore, the Keller-Segel system on $(\mathbb{H}^n,{\mathfrak g})$ is rewritten as
\begin{equation}\label{smKS1} 
u_t =  {\Delta_{\mathbb{H}^n}} u + \nabla_x \cdot \left[ - \alpha u\nabla_x (- {\Delta_{\mathbb{H}^n}}+ \gamma I)^{-1}u + f\right] \quad  x\in  \mathbb{H}^n,\, t \in \r. 
\end{equation}

Setting $L_j = \partial_j(- {\Delta_{\mathbb{H}^n}} + \gamma I)^{-1}$, the properties of this operator is given in the following lemma (see \cite[Lemma 4.1]{Fe2021} and \cite[Lemma 3.3]{Pi}):
\begin{lemma}\label{invertEs}
Let $\gamma\geqslant 0$, $n\geqslant 2$, $1<p<n$ and $\frac{1}{q}=\frac{1}{p}-\frac{1}{n}$.
The operator $L_j$ is continuous from $L^{p}(\mathbb{H}^n)$ to $L^{q}(\mathbb{H}^n)$, for each $j=1,2...n$. Moreovver, there exists a constant $C>0$ independent of $f$ and $\gamma$ satisfying 
\begin{equation}
\norm{L_jf}_{L^{q}} \leqslant C{k(\gamma)}\norm{f}_{L^{p}},
\end{equation}
where ${k(0)}=1$ and ${k(\gamma)}=\gamma^{-(n-1)}$ if $\gamma>0$.
\end{lemma}

The dispersive and smoothing estimates of heat semigroup on hyperbolic space are well
studied in the literature for hyperbolic spaces and various types of non-compact manifolds for both the Laplace-Beltrami and the Hodge-Laplacian {operators} (see \cite{Anker1,Auscher,Bakry,Co,Da,Gri,Pi,Va, Vaz} and many others). Here, we will use the estimates of Pierfelice.
\begin{lemma}(\cite[Theorem 4.1 and Corollary 4.3]{Pi})\label{estimates}
\begin{itemize}
\item[(i)] For $t>0$, and $p$, $q$ such that $1\leq p \leq q \leq \infty$, 
the following dispersive estimate holds: 
\begin{equation}\label{dispersive}
\left\| e^{t {\Delta_{\mathbb{H}^n}}}u_0\right\|_{L^q(\mathbb{H}^n)} \leq [h_n(t)]^{\frac{1}{p}-\frac{1}{q}}e^{-t( \gamma_{p,q})}\left\|u_0\right\|_{L^p(\mathbb{H}^n)} 
\end{equation}
for all $u_0 \in L^p(\mathbb{H}^n,\mathbb{R})$, where 
 $$h_n(t) = \tilde{C}\max\left( \frac{1}{t^{n/2}},1 \right),\, 
   \gamma_{p,q}=\frac{\delta_n}{2}\left[ \left(\frac{1}{p} - \frac{1}{q} \right) + \frac{8}{q}\left( 1 - \frac{1}{p} \right) \right]$$ 
and $\delta_n$ is a positive constant depending only on $n$.  
\item[(ii)] For $t>0$, and $p,q$ such that $1\leqslant p\leqslant q \leqslant\infty$, the following estimate holds:
\begin{equation}
\left\|  {\nabla_x\cdot e^{t\Delta_{\mathbb{H}^n}}} V_0 \right\|_{L^q(\mathbb{H}^n)} \leqslant [h_n(t)]^{\frac{1}{p}-\frac{1}{q}+\frac{1}{n}}e^{-t\left( \frac{\gamma_{q,q}+\gamma_{p,q}}{2} \right)} \left\|V_0\right\|_{L^p(\mathbb{H}^n)}
\end{equation}
for all vector field $V_0 \in L^p(\mathbb{H}^n)$. The functions $h_n(t)$ and $\gamma_{p,q}$ are defined as in Assertion (i).
\end{itemize}
\end{lemma}
Using the estimates in Lemma \ref{estimates}, Pierfelice and Maheux proved the well-posedness of local and global mild solutions of \eqref{KSH} on the two-dimensional hyperbolic space $\mathbb{H}^2$ (see \cite{MaPi}). Next, we will show the well-posedness of global mild (asymptotically almost periodic) solutions for all dimensions $n\geqslant 2$.

\section{The existence of asymptotically almost periodic solution to the  linear Keller-Segel equation}\label{S3}
In this section, we concentrate on studying the inhomogeneous  linear equation corresponding to Equation \eqref{smKS1} is
\begin{equation}\label{LinearKS} 
u_t =  {\Delta_{\mathbb{H}^n}} u + \nabla_{x} \cdot \left[ -\alpha \omega\nabla_x (-  {\Delta_{\mathbb{H}^n}} + \gamma I)^{-1}\omega + f\right], \quad  x\in  \mathbb{H}^n,\, t \in\r,
\end{equation}
for a given function $\omega \in C_b(\mathbb{R}, L^{\frac{n}{2}}(\mathbb{H}^n))$.
The $\textit{mild solution}$ on the whole  {line time-axis} can be reformulated by the following integral equation (see \cite{KoNa}):
\begin{equation}\label{mild:linear}
u(t) = \int_{-\infty}^t \nabla_x \cdot e^{(t-s) {\Delta_{\mathbb{H}^n}}}\left[-\alpha \omega\nabla_x(- {\Delta_{\mathbb{H}^n}} + \gamma I)^{-1}\omega + f\right](s)ds.
\end{equation}

In the case of {asymptotically} almost periodic functions, we need to add the initial data $u(0)=u_0$ to the equation \eqref{LinearKS}. That yields the Cauchy problem as
\begin{equation}\label{LKS} 
\left\{
  \begin{array}{rll}
u_t \!\! &= {\Delta_{\mathbb{H}^n}} u +\nabla_{x} \cdot \left[ -\alpha \omega\nabla_x (- {\Delta_{\mathbb{H}^n}} + \gamma I)^{-1}\omega + f\right],  \quad  &x\in  \mathbb{H}^n, t \in \r_+,\hfill \cr
u(0) \!\!&=\; u_0, \quad &x\in \mathbb{H}^n.
\end{array}\right.
\end{equation}
By Duhamel’s principle, the mild solution is given by
\begin{equation}\label{mild:linear*}
u(t) = e^{t {\Delta_{\mathbb{H}^n}} }u_0 + \int_0^t \nabla_x \cdot e^{(t-s) {\Delta_{\mathbb{H}^n}}}\left[ -\alpha \omega\nabla_x(- {\Delta_{\mathbb{H}^n}}+ \gamma I)^{-1}\omega + f \right](s)ds.
\end{equation}
The existence and uniqueness of  the bounded mild solutions of the inhomogeneous linear system \eqref{LinearKS} is established in the following lemma.
 
\begin{lemma}\label{Thm:linear}
Let $n\geq 2$ and ${\max \left\lbrace 3,n\right\rbrace} <p<2n$. The following assertions holds
\begin{itemize}
\item[(i)] For given functions $\omega \in C_b(\mathbb{R}, L^{\frac{p}{2}}(\mathbb{H}^n))$ and $f\in C_b(\r, L^{\frac{p}{3}}(\mathbb{H}^n))$, there exists a unique mild solution of Equation \eqref{LinearKS} satisfying the integral equation \eqref{mild:linear}. Moreover, the following boundedness holds
\begin{equation}\label{boundedness12}
\norm{u(t)}_{L^{\frac{p}{2}}(\mathbb{H}^n)} \leq  \tilde{K}\left( \alpha {k(\gamma)}\norm{\omega}^2_{\infty,L^{\frac{p}{2}}(\mathbb{H}^n)} + \norm{f}_{\infty,L^{\frac{p}{3}}(\mathbb{H}^n)} \right).
\end{equation}
\item[(ii)] For given functions $u_0\in L^{\frac{p}{2}}(\mathbb{H}^n)$, $\omega \in C_b(\mathbb{R}_+, L^{\frac{p}{2}}(\mathbb{H}^n))$ and $f\in C_b(\r_+, L^{\frac{p}{3}}(\mathbb{H}^n))$, there exists a unique mild solution of Equation \eqref{LKS} with the initial data $u_0$ satisfying the integral equation \eqref{mild:linear*}. Moreover, the following boundedness holds
\begin{equation}\label{boundedness1}
\norm{u(t)}_{L^{\frac{p}{2}}(\mathbb{H}^n)} \leq  \norm{u_0}_{L^{\frac{p}{2}}(\mathbb{H}^n)} + \tilde{K}\left( \alpha {k(\gamma)} \norm{\omega}^2_{\infty,L^{\frac{p}{2}}(\mathbb{H}^n)} + \norm{f}_{\infty,L^{\frac{p}{3}}(\mathbb{H}^n)} \right).
\end{equation}
\end{itemize}
\end{lemma}
\begin{proof} 
$(i)$ By a utilization of Lemma \ref{estimates} and the boundedness of $L_j = \partial_j(-\Delta_{\mathbb{H}^n} + \gamma I)^{-1}$, we have
\begin{eqnarray*}
\norm{u(t)}_{L^{\frac{p}{2}}(\mathbb{H}^n)} &\leq&  \alpha \int_{-\infty}^t \norm{\nabla_x \cdot e^{(t-s) {\Delta_{\mathbb{H}^n}}}\left[\omega\nabla_x(- {\Delta_{\mathbb{H}^n}} + \gamma I)^{-1}\omega \right](s)}_{L^{\frac{p}{2}}(\mathbb{H}^n)}ds\cr
&&+ \int_{-\infty}^t \norm{\nabla_x \cdot e^{(t-s) {\Delta_{\mathbb{H}^n}}}f(s)}_{L^{\frac{p}{2}}(\mathbb{H}^n)}ds\cr
&\leq&  \alpha \int_{-\infty}^t [h_n(t-s)]^{\frac{2}{p}}e^{-(t-s)\beta}\norm{\left[ \omega\nabla_x(- {\Delta_{\mathbb{H}^n}} + \gamma I)^{-1}\omega \right](s)}_{L^{\frac{pn}{4n-p}}(\mathbb{H}^n)}ds\cr
&&+ \int_{-\infty}^t [h_n(t-s)]^{\frac{1}{p}+\frac{1}{n}}e^{-(t-s)\hat\beta}\norm{f(s)}_{L^{\frac{p}{3}}(\mathbb{H}^n)}ds\cr
&\leq&  \alpha \int_{-\infty}^t [h_n(t-s)]^{\frac{2}{p}}e^{-(t-s)\beta}\norm{\omega(s)}_{L^{\frac{p}{2}}(\mathbb{H}^n)}  \norm{\left[\nabla_x(- {\Delta_{\mathbb{H}^n}} + \gamma I)^{-1}\omega \right](s)}_{L^{\frac{pn}{2n-p}}(\mathbb{H}^n)}ds\cr
&&+ \int_{-\infty}^t [h_n(t-s)]^{\frac{1}{p}+\frac{1}{n}}e^{-(t-s)\hat\beta}\norm{f(s)}_{L^{\frac{p}{3}}(\mathbb{H}^n)}ds\cr
&\leqslant& \int_{-\infty}^t \left( (t-s)^{-\frac{n}{p}} +1 \right) e^{-(t-s)\beta}ds \left( \alpha \tilde C^{\frac{2}{p}} C {k(\gamma)}\norm{\omega}^2_{\infty, L^{\frac{p}{2}}(\mathbb{H}^n))} \right)\cr
&&+ \int_{-\infty}^t \left( (t-s)^{-\frac{1}{2}-\frac{n}{2p}} +1 \right) e^{-(t-s)\hat\beta}ds \left(\tilde C^{\frac{1}{p}+\frac{1}{n}} \norm{f}_{\infty,L^{\frac{p}{3}}(\mathbb{H}^n)}\right)\cr
&\leq& \int_0^\infty \left( z^{-\frac{n}{p}} +1 \right) e^{-z\beta}dz \left( \alpha \tilde C^{\frac{2}{p}} C {k(\gamma)}\norm{\omega}^2_{\infty, L^{\frac{p}{2}}(\mathbb{H}^n))} \right)\cr
&&+ \int_0^\infty \left( z^{-\frac{1}{2}-\frac{n}{2p}} +1 \right) e^{-z\hat\beta}dz \left(\tilde C^{\frac{1}{p}+\frac{1}{n}} \norm{f}_{\infty,L^{\frac{p}{3}}(\mathbb{H}^n)}\right)\cr
&\leq&  \left( \frac{1}{\beta^{1- \frac{n}{p}}}\Gamma\left(1- \frac{n}{p} \right) + \frac{1}{\beta} \right)\left(\alpha \tilde C^{\frac{2}{p}} C{k(\gamma)}  \norm{\omega}^2_{\infty, L^{\frac{p}{2}}(\mathbb{H}^n)} \right)\cr
&&+  \left(\frac{1}{\hat \beta^{\frac{1}{2} - \frac{n}{2p}}}\Gamma\left( \frac{1}{2} - \frac{n}{2p} \right) + \frac{1}{\hat \beta} \right) \tilde C^{\frac{1}{p}+\frac{1}{n}} \norm{f}_{\infty,L^{\frac{p}{3}}(\mathbb{H}^n)} \cr
&\leqslant& \tilde{K}\left(\alpha {k(\gamma)}   \norm{\omega}^2_{\infty,L^{\frac{p}{2}}(\mathbb{H}^n)} + \norm{f}_{\infty,L^{\frac{p}{3}}(\mathbb{H}^n)}\right),
\end{eqnarray*}
 where  $\beta=\dfrac{\gamma_{p/2,p/2}+\gamma_{pn/(4n-p),p/2}}{2}$,\; $\hat \beta=\dfrac{\gamma_{p/2,p/2}+\gamma_{p/3,p/2}}{2}$, $$\tilde{K}=\max\left\lbrace  \left( \frac{1}{\beta^{1- \frac{n}{p}}}\Gamma\left(1- \frac{n}{p} \right) + \frac{1}{\beta} \right)  \tilde C^{\frac{2}{p}} C, \left( \frac{1}{\hat\beta^{\frac{1}{2} - \frac{n}{2p}}}\Gamma\left( \frac{1}{2} -\frac{n}{2p} \right) + \frac{1}{\hat\beta}\right) \tilde C^{\frac{1}{p}+\frac{1}{n}}\right\rbrace,$$ and  $\mathit{\mathbf{\Gamma}}$ means the Gamma function.\\
 
$(ii)$ By the same way in the proof of Assertion $(i)$, we estimate that 
\begin{eqnarray*}
\norm{u(t)}_{L^{\frac{p}{2}}(\mathbb{H}^n)}\!\!\!\!\!\! &\leq& \!\!\! \norm{u_0}_{L^{\frac{p}{2}}(\mathbb{H}^n)}   + \alpha \int_0^t \norm{\nabla_x \cdot e^{(t-s) {\Delta_{\mathbb{H}^n}}}\left[ \omega\nabla_x(- {\Delta_{\mathbb{H}^n}} + \gamma I)^{-1}\omega \right](s)}_{L^{\frac{p}{2}}(\mathbb{H}^n)}ds\cr
&&\!\!\!+ \int_0^t \norm{\nabla_x \cdot e^{(t-s) {\Delta_{\mathbb{H}^n}}}f(s)}_{L^{\frac{p}{2}}(\mathbb{H}^n)}ds\cr
&\leq&\!\!\! \norm{u_0}_{L^{\frac{p}{2}}(\mathbb{H}^n)} + \alpha \int_0^t [h_n(t-s)]^{\frac{2}{p}}e^{-(t-s)\beta}\norm{\left[ \omega\nabla_x(- {\Delta_{\mathbb{H}^n}} + \gamma I)^{-1}\omega \right](s)}_{L^{\frac{pn}{4n-p}}(\mathbb{H}^n)}ds\cr
&&\!\!\!+ \int_0^t [h_n(t-s)]^{\frac{1}{p}+\frac{1}{n}}e^{-(t-s)\hat\beta}\norm{f(s)}_{L^{\frac{p}{3}}(\mathbb{H}^n)}ds\cr
&\leq&\!\!\! \norm{u_0}_{L^{\frac{p}{2}}(\mathbb{H}^n)} + \alpha \int_0^t [h_n(t-s)]^{\frac{2}{p}}e^{-(t-s)\beta}\norm{\omega(s)}_{L^{\frac{p}{2}}(\mathbb{H}^n)}  \norm{\left[\nabla_x(- {\Delta_{\mathbb{H}^n}} + \gamma I)^{-1}\omega \right](s)}_{L^{\frac{pn}{2n-p}}(\mathbb{H}^n)}ds\cr
&&\!\!\!+ \int_0^t [h_n(t-s)]^{\frac{1}{p}+\frac{1}{n}}e^{-(t-s)\hat\beta}\norm{f(s)}_{L^{\frac{p}{3}}(\mathbb{H}^n)}ds\cr
&\leq&\!\!\! \norm{u_0}_{L^{\frac{p}{2}}(\mathbb{H}^n)} + \int_0^t \left( (t-s)^{-\frac{n}{p}} +1 \right) e^{-(t-s)\beta}ds \left( \alpha \tilde C^{\frac{2}{p}} C {k(\gamma)}\norm{\omega}^2_{\infty, L^{\frac{p}{2}}(\mathbb{H}^n))} \right)\cr
&&\!\!\!+ \int_0^t \left( (t-s)^{-\frac{1}{2}-\frac{n}{2p}} +1 \right) e^{-(t-s)\hat\beta}ds \left(\tilde C^{\frac{1}{p}+\frac{1}{n}} \norm{f}_{\infty,L^{\frac{p}{3}}(\mathbb{H}^n)}\right)\cr
&\leq&\!\!\! \norm{u_0}_{L^{\frac{p}{2}}(\mathbb{H}^n)} + \left( \frac{1}{\beta^{1- \frac{n}{p}}}\Gamma\left(1- \frac{n}{p} \right) + \frac{1}{\beta} \right)\left(\alpha \tilde C^{\frac{2}{p}} C {k(\gamma)}  \norm{\omega}^2_{\infty, L^{\frac{p}{2}}(\mathbb{H}^n)} \right)\cr
&&\!\!\!+  \left(\frac{1}{\hat \beta^{\frac{1}{2} - \frac{n}{2p}}}\Gamma\left( \frac{1}{2} - \frac{n}{2p} \right) + \frac{1}{\hat \beta} \right) \tilde C^{\frac{1}{p}+\frac{1}{n}} \norm{f}_{\infty,L^{\frac{p}{3}}(\mathbb{H}^n)} \cr
&\leq&\!\!\! \norm{u_0}_{L^{\frac{p}{2}}(\mathbb{H}^n)} + \int_0^\infty \left( z^{-\frac{n}{p}} +1 \right) e^{-z\beta}dz \left( \alpha \tilde C^{\frac{2}{p}} C {k(\gamma)}\norm{\omega}^2_{\infty, L^{\frac{p}{2}}(\mathbb{H}^n))} \right)\cr
&&\!\!\!+ \int_0^\infty \left( z^{-\frac{1}{2}-\frac{n}{2p}} +1 \right) e^{-z\hat\beta}dz \left(\tilde C^{\frac{1}{p}+\frac{1}{n}} \norm{f}_{\infty,L^{\frac{p}{3}}(\mathbb{H}^n)}\right)\cr
&\leq&\!\!\! \norm{u_0}_{L^{\frac{p}{2}}(\mathbb{H}^n)} + \tilde{K}\left(\alpha {k(\gamma)}   \norm{\omega}^2_{L^\infty(\r_+, L^{\frac{p}{2}}(\mathbb{H}^n))} + \norm{f}_{\infty,L^{\frac{p}{3}}(\mathbb{H}^n)}\right),
\end{eqnarray*} 
where $\widetilde{K}$ is given as in the proof of Assertion $(i)$. Our proof is completed.
\end{proof} 
For the sake of convenience, we denote $X=  L^{\frac{p}{3}}(\mathbb{H}^n)$ and $Y= L^{\frac{p}{2}}(\mathbb{H}^n)$. By using Assertion $(ii)$ of Lemma \ref{Thm:linear}, we can define the solution operator $S: C_b(\r_+, X\times Y)\to \mathcal{X}=C_b(\r_+,Y)$ of system \eqref{LKS} as follows
\begin{align*}
S: C_b(\r_+,X\times Y) &\rightarrow \mathcal{X}=C_b(\r_+,Y)\cr
(f,\omega)&\mapsto S(f,\omega)
\end{align*}
where $X\times Y$ is the Cartesian product space equipped with the norm $\norm{\cdot}_{X\times Y} = \norm{\cdot}_X + \norm{\cdot}_Y$ and
\begin{equation}\label{SolOpe}
S(f,\omega)(t) = e^{t {\Delta_{\mathbb{H}^n}}}u_0 + \int_0^t \nabla_x \cdot e^{(t-s) {\Delta_{\mathbb{H}^n}}}\left[ -\alpha \omega\nabla_x(- {\Delta_{\mathbb{H}^n}} + \gamma I)^{-1}\omega+ f \right](s)ds.
\end{equation}
The norm of $AAP(\r_+,X\times Y)$ is inherited from the above norm.

Next, we state and prove the main theorem of this section.
\begin{theorem}\label{pest}
Let $n\geq 2$ and ${\max \left\lbrace 3,n\right\rbrace } <p<2n$. Suppose that functions $f$ and $\omega$ are given such that the function $t \mapsto (f(t),\omega(t))$ belongs  to $AAP(\mathbb{R}_+,X\times Y)$. Then, the Cauchy problem \eqref{LKS} has one and only one mild solution $\hat{u} \in AAP(\mathbb{R}_+,Y)$ satisfying 
\begin{equation}\label{boundedness2}
\norm{u(t)}_{L^{\frac{p}{2}}(\mathbb{H}^n)} \leq  \norm{u_0}_{L^{\frac{p}{2}}(\mathbb{H}^n)}  +  \tilde{K}\left( \alpha {k(\gamma)}\norm{\omega}^2_{\infty,L^{\frac{p}{2}}(\mathbb{H}^n)} + \norm{f}_{\infty,L^{\frac{p}{3}}(\mathbb{H}^n)} \right),
\end{equation}
where $\tilde{K}$ is given in Lemma \ref{Thm:linear}.
\end{theorem} 
\def\xcal{\mathcal X}
\begin{proof}

Thanks to Lemma \ref{Thm:linear}, it is {sufficient} to show that the solution operator $S$ maps $AAP(\mathbb{R}_+,X\times Y)$ into $AAP(\mathbb{R}_+,Y)$. 

Indeed, for each $(f,\omega)\in AAP( \mathbb{R}_+, X\times Y) $, there exist  $(H,\eta) \in AP(\mathbb{R},X\times Y)$ and $ (\Phi,\theta) \in C_0(\r_+,X\times Y)$ such that  $f(t) =  H(t)+ \Phi(t)$ and $\omega(t)=\eta(t)+\theta(t)$ for all $t \in \r_+$. Utilization of \eqref{SolOpe}, we have
\begin{eqnarray*}\label{seperate}
&S&\!\!\!\!\!\!(f,\omega)(t) = e^{t {\Delta_{\mathbb{H}^n}} }u_0 + \int_0^t \nabla_x \cdot e^{(t-s) {\Delta_{\mathbb{H}^n}}}\left[ -\alpha \omega\nabla_x(- {\Delta_{\mathbb{H}^n}} + \gamma I)^{-1}\omega+ f \right](s)ds\cr
&=& e^{t {\Delta_{\mathbb{H}^n}}}u_0 + \int_0^t \nabla_x \cdot e^{(t-s) {\Delta_{\mathbb{H}^n}}}H(s)ds +  \int_0^t \nabla_x \cdot e^{(t-s) {\Delta_{\mathbb{H}^n}}}\Phi(s)ds \cr
&&-  \int_0^t \nabla_x \cdot e^{(t-s) {\Delta_{\mathbb{H}^n}}}\left[ \alpha \eta\nabla_x(- {\Delta_{\mathbb{H}^n}}+ \gamma I)^{-1}\eta \right](s)ds - \int_0^t \nabla_x \cdot e^{(t-s) {\Delta_{\mathbb{H}^n}}}\left[ \alpha \theta\nabla_x( -{\Delta_{\mathbb{H}^n}}+ \gamma I)^{-1}\omega \right](s)ds \cr
&&-  \int_0^t \nabla_x \cdot e^{(t-s) {\Delta_{\mathbb{H}^n}}}\left[ \alpha \eta\nabla_x(- {\Delta_{\mathbb{H}^n}}+ \gamma I)^{-1}\theta \right](s)ds  \cr
&=& e^{t {\Delta_{\mathbb{H}^n}} }u_0 + \int_{-\infty}^t \nabla_x \cdot e^{(t-s) {\Delta_{\mathbb{H}^n}}}H(s)ds +  \int_0^t \nabla_x \cdot e^{(t-s) {\Delta_{\mathbb{H}^n}}}\Phi(s)ds - \int_{-\infty}^0 \nabla_x \cdot e^{(t-s) {\Delta_{\mathbb{H}^n}}}H(s)ds\cr
&&-  \int_{\infty}^t \nabla_x \cdot e^{(t-s) {\Delta_{\mathbb{H}^n}}}\left[ \alpha \eta\nabla_x(- {\Delta_{\mathbb{H}^n}}+ \gamma I)^{-1}\eta \right](s)ds - \int_0^t \nabla_x \cdot e^{(t-s) {\Delta_{\mathbb{H}^n}}}\left[ \alpha \theta\nabla_x(-{\Delta_{\mathbb{H}^n}}+ \gamma I)^{-1}\omega \right](s)ds  \cr
&& - \int_0^t \nabla_x \cdot e^{(t-s) {\Delta_{\mathbb{H}^n}}}\left[\alpha \eta\nabla_x(-{\Delta_{\mathbb{H}^n}}+ \gamma I)^{-1}\theta \right](s)ds +\int_{\infty}^0 \nabla_x \cdot e^{(t-s) {\Delta_{\mathbb{H}^n}}}\left[ \alpha \eta\nabla_x(- {\Delta_{\mathbb{H}^n}}+ \gamma I)^{-1}\eta \right](s)ds \cr
\end{eqnarray*}
for all $t\in \mathbb{R}_+$.
Setting
\begin{eqnarray*}
&&\hat{S}(H)(t)=\int_{-\infty}^t \nabla_x \cdot e^{(t-s){\Delta_{\mathbb{H}^n}}}H(s)ds ,\cr
&& \hat{S}(\eta, \eta)(t)=\int_{-\infty}^t \nabla_x \cdot e^{(t-s) {\Delta_{\mathbb{H}^n}}}\left[ \alpha \eta\nabla_x(- {\Delta_{\mathbb{H}^n}}+ \gamma I)^{-1}\eta \right](s)ds,
\end{eqnarray*}
and
\begin{eqnarray*}
&&\tilde{S}(\Phi)(t) = \int_{0}^t  \nabla_x \cdot e^{(t-s) {\Delta_{\mathbb{H}^n}}}\Phi(s)ds ,\cr
&&\tilde{S}(\theta, \omega)(t) = \int_{0}^t  \nabla_x \cdot e^{(t-s) {\Delta_{\mathbb{H}^n}}}\left[ \alpha \theta\nabla_x(-{\Delta_{\mathbb{H}^n}}+ \gamma I)^{-1}\omega \right](s)ds,\cr
&&\tilde{S}(\eta,\theta)(t) = \int_{0}^t  \nabla_x \cdot e^{(t-s) {\Delta_{\mathbb{H}^n}}}\left[ \alpha \eta\nabla_x(-{\Delta_{\mathbb{H}^n}}+ \gamma I)^{-1}\theta \right](s)ds.
\end{eqnarray*}
Following Lemma \ref{Thm:linear}, we have that: the functions $\hat{S}(H)$ and $\hat{S}(\eta, \eta)$ are bounded in $C_b(\mathbb{R},Y)$; the functions  $\tilde{S}(\Phi),\,\tilde{S}(\theta, \theta)$, $\tilde{S}(\eta, \theta)$, $\tilde{S}(\theta, \eta)$ and {$\tilde{S}(\theta,\omega )$} are bounded in $C_b(\mathbb{R}_+,Y)$. {From} \eqref{seperate}, we have
\begin{eqnarray}
S(f,\omega)(t) &=& e^{t\Delta_{\mathbb{H}^n} }u_0 + \left( \hat{S}(H)(t) - \hat{S}(\eta, \eta)(t) \right)\cr
&&+ \left( \tilde{S}(\Phi)(t)  -\tilde{S}( \theta, \omega)(t) - \tilde{S}( \eta, \theta ) \right) \cr
&&- \left( e^{-t\Delta_{\mathbb{H}^n}}\hat{S}(H)(0) - e^{-t\Delta_{\mathbb{H}^n}}\hat{S}(\eta, \eta)(0) \right) \hbox{   for all  } t\in\mathbb{R_+}.
\end{eqnarray}

Now, we prove $S(f,\omega)\in AAP(\r_+,Y)$ by  three below steps:\\

\underline{\bf Step 1:} We first show that the function $ \hat{S}(H)-\hat{S}(\eta, \eta)$ belongs to $AP(\mathbb{R}, Y)$. 
Since $ (H,\eta)\in AP(\mathbb{R}, X\times Y)$, for each $ \epsilon  > 0$, there exists $l_{\epsilon}>0 $ such that every interval of length $l_{\epsilon}$ contains at least a number $T $ with the following property
$$\sup_{t \in \r } \left( \| H(t+T)  - H(t) \|_{C_b(\r,X)} + \| \eta(t+T) - \eta(t) \|_{C_b(\r,Y)} \right) < \epsilon,$$
where
$\norm{H}_{C_b(\r,X)}= \sup_{t\in \mathbb{R}} \norm{H(t)}_{X} $
and $\norm{\eta}_{C_b(\r,Y)}= \sup_{t\in \mathbb{R}}\norm{\eta(t)}_{Y}$.\\
By the same way as the proof of Lemma \ref{Thm:linear}, we can estimate
\begin{eqnarray}\label{AP1}
&&\left\|\hat{S}(H)(t+T) - \hat{S}(H)(t)\right\|_Y \cr
&&= \left\|\int_{-\infty}^{t+T} \nabla_x \cdot e^{(t+T-s) {\Delta_{\mathbb{H}^n}}}H(s)
ds - \int_{-\infty}^{t} \nabla_x \cdot e^{(t-s) {\Delta_{\mathbb{H}^n}}}H(s) ds \right\|_Y \cr
&&= \left\|\int_{0}^\infty  \nabla_x \cdot e^{s {\Delta_{\mathbb{H}^n}}}[ H(t+T-\tau)- H(t-\tau)] ds \right\|_Y \cr
&&\leq \tilde{K}\left\| H(\cdot+T) - H(\cdot)\right\|_{C_b(\r,X)} < \tilde{K}\epsilon,
\end{eqnarray}
for all $t \in \r$, where $\tilde{K}$ is determined in Lemma \ref{Thm:linear}. 
Moreover, we also have
\begin{eqnarray*}
&\hat{S}&\!\!\!\!\!(\eta,\eta)(t+T) - \hat{S}(\eta, \eta)(t)\cr
&=& \int_{-\infty}^{t+T} \nabla_x \cdot e^{(t+T-s) {\Delta_{\mathbb{H}^n}}}\left[ \alpha \eta\nabla_x(- {\Delta_{\mathbb{H}^n}}+ \gamma I)^{-1}\eta \right](s)ds\cr
&& - \int_{-\infty}^{t} \nabla_x \cdot e^{(t-s) {\Delta_{\mathbb{H}^n}}}\left[ \alpha \eta\nabla_x(- {\Delta_{\mathbb{H}^n}}+ \gamma I)^{-1}\eta \right](s)ds \cr
&=& \int_{0}^\infty \nabla_x \cdot e^{s {\Delta_{\mathbb{H}^n}}} \{\left[ \alpha \eta\nabla_x(- {\Delta_{\mathbb{H}^n}}+ \gamma I)^{-1}\eta \right](t+T-s) -\left[ \alpha \eta\nabla_x(- {\Delta_{\mathbb{H}^n}}+ \gamma I)^{-1}\eta \right](t-s)  \} ds.\cr
\end{eqnarray*}
Thus,
\begin{eqnarray}\label{AP2}
&&\left\|\hat{S}(\eta,\eta)(t+T) - \hat{S}(\eta,\eta)(t)\right\|_Y \cr
&&\leq \left\| \int_{0}^\infty  \nabla_x \cdot e^{s {\Delta_{\mathbb{H}^n}}}  \alpha \left[ \eta (t+T-s)- \eta (t-s)) \right] \nabla_x(-{\Delta_{\mathbb{H}^n}}+ \gamma I)^{-1}\eta (t+T-s)  ds \right\|_Y\cr
&&+ \left\|\int_{0}^\infty  \nabla_x \cdot e^{s {\Delta_{\mathbb{H}^n}}}  \alpha  \eta (t-s))  \nabla_x(- {\Delta_{\mathbb{H}^n}}+ \gamma I)^{-1}\left[ \eta (t+T-s)- \eta (t-s)) \right]  ds \right\|_{Y} \cr
&&\leq 2\tilde{K}\alpha {k(\gamma)} \norm{\eta(\cdot+T)- \eta(\cdot)}_{C_b(\r,Y)}\norm{\eta}_{C_b(\r,Y)} < 2\tilde{K}\alpha {k(\gamma)}\norm{\eta}_{C_b(\r,Y)}\epsilon,
\end{eqnarray}
for all $t\in \mathbb{R}$, where $\tilde{K}$ is given as in Lemma \ref{Thm:linear}.

Combining inequalities \eqref{AP1} and \eqref{AP2}, we obtain that 
$$ \norm{(\hat{S}(H)-\hat{S}(\eta, \eta))(t+T) - (\hat{S}(H)-\hat{S}(\eta, \eta))(t)}_{Y} < \tilde{K}(1+  2\alpha {k(\gamma)}\norm{\eta}_{C_b(\r,Y)})\epsilon.$$
Consequently, $\hat{S}(H)-\hat{S}(\eta, \eta)$ belongs to $AP(\mathbb{R}, Y)$.\\

\underline{\bf Step 2:} Secondly, we show that $\tilde{S}(\Phi) - \tilde{S}(\theta, \omega) - \tilde{S}(\eta, \theta)$ belongs to $C_0(\r_+,Y)$. 
Indeed, the first term $\tilde{S}(\Phi)(t)$ can be rewritten as follows
\begin{eqnarray*}
\tilde{S}(\Phi)(t) &=& \int_{0}^t  \nabla_x \cdot e^{(t-s) {\Delta_{\mathbb{H}^n}}}\Phi(s)ds \hbox{   } (t\in\mathbb{R_+}) \cr
&=&\int_{0}^{t/2}  \nabla_x \cdot e^{(t-s) {\Delta_{\mathbb{H}^n}}}\Phi(s)ds  +\int_{t/2}^t  \nabla_x \cdot e^{(t-s) {\Delta_{\mathbb{H}^n}}}\Phi(s)ds \cr
&=& S_1(\Phi)(t) +S_2(\Phi)(t).
\end{eqnarray*}
By the similar proof of Lemma \ref{Thm:linear}, for $t>2$, one implies
\begin{eqnarray*}
\norm{S_1(\Phi)(t)}_Y &\leqslant& \int_0^{t/2} \left( (t-s)^{-\frac{1}{2}-\frac{n}{2p}} +1 \right) e^{-(t-s)\hat\beta}ds \left(\tilde C^{\frac{1}{p}+\frac{1}{n}} \norm{\Phi}_{\infty,L^{\frac{p}{3}}(\mathbb{H}^n)}\right)\cr
&\leq&  \tilde C^{\frac{1}{p}+\frac{1}{n}} \left( \left( \frac{t}{2}\right)^{-\frac{1}{2}-\frac{n}{2p}}+1\right)\frac{1}{\hat\beta}(e^{-\frac{\hat\beta t}{2}} - e^{-\hat\beta t}) \norm{\Phi}_{C_b(\r_+,X)}\cr
& \to & 0, \hbox{   when  } t \to +\infty.
\end{eqnarray*}
And hence,
\begin{equation}\label{lim11}
\lim_{t\to +\infty}\norm{S_1(\Phi)(t)}_{Y}=0.
\end{equation}

Further note that $\lim\limits_{t\to +\infty}\norm{\Phi(t)}_X = 0$, then for all $\epsilon>0$, there exists $t_0>0$ large enough such that for all $t>t_0$, we have $\norm{\Phi(t)}_X<\epsilon$.
Using this and again by the same way to prove Lemma \ref{Thm:linear}, it therefore follows that
\begin{eqnarray*}
\norm{S_2(\Phi)(t)}_{Y} \leqslant \tilde{K}  \epsilon \hbox{   for  } t>t_0,
\end{eqnarray*}
where the constant $\tilde{K}$ is given in Lemma \ref{Thm:linear}. This yields 
\begin{equation}\label{lim12}
\lim_{t\to+\infty}\norm{S_2(\Phi)(t)}_{Y}=0.
\end{equation}
Combining \eqref{lim11} with \eqref{lim12}, we obtain
\begin{equation}\label{lim1}
\lim_{t\to+\infty}\norm{\tilde{S}(\Phi)(t)}_{Y}=0.
\end{equation}
Consequently, $\tilde{S}(\Phi)\in C_0(\r_+,Y)$.

 Only the case $\tilde{S}(\theta,\omega)\in C_0(\r_+,Y)$ is considered. The case $\tilde{S}(\eta, \theta)\in C_0(\r_+,Y)$ can be treated by a similar approach. Indeed, we have 
\begin{eqnarray*}
&\tilde{S}&\!\!\!\!\!(\theta,\omega)(t)\cr&=& \int_{0}^t \nabla_x \cdot e^{(t-s) {\Delta_{\mathbb{H}^n}}}\left[ \alpha \theta\nabla_x({-\Delta_{\mathbb{H}^n}}+ \gamma I)^{-1}\omega \right](s)ds \hbox{   } (t\in\mathbb{R_+}) \cr
&=&\int_{0}^{t/2}  \nabla_x \cdot e^{(t-s) {\Delta_{\mathbb{H}^n}}}\left[ \alpha \theta\nabla_x( -{\Delta_{\mathbb{H}^n}}+ \gamma I)^{-1}\omega \right](s)ds +\int_{t/2}^t \nabla_x \cdot e^{(t-s) {\Delta_{\mathbb{H}^n}}}\left[ \alpha \theta\nabla_x(-{\Delta_{\mathbb{H}^n}}+ \gamma I)^{-1}\omega \right](s)ds\cr
&=& S_3(\theta,\omega)(t) +S_4(\theta,\omega)(t).
\end{eqnarray*}
Furthermore, the argument from the proof of Lemma \ref{Thm:linear}, for $t>2$, can be reused to deduce 
\begin{eqnarray*}
 \|S_3(\theta,\omega)(t)\|_{Y} &\leq& \int_0^{t/2} \|\nabla_x \cdot e^{(t-s) {\Delta_{\mathbb{H}^n}}}\left[ \alpha \theta\nabla_x(-{\Delta_{\mathbb{H}^n}}+ \gamma I)^{-1}\omega \right](s)\|_{Y}ds\cr 
&\le& \int_0^{t/2} [h_n(t-s)]^{\frac{2}{p}}e^{-\beta(t-s)}\alpha   C {k(\gamma)}ds \norm{\theta}_{\infty, L^{\frac{p}{2}}(\mathbb{H}^n)} \norm{\omega}_{\infty, L^{\frac{p}{2}}(\mathbb{H}^n)} \cr
&\le& \int_0^{t/2} e^{-\beta(t-s)}ds  \alpha \tilde C^{\frac{2}{p}} C{k(\gamma)}\norm{\theta}_{\infty, L^{\frac{p}{2}}(\mathbb{H}^n)} \norm{\omega}_{\infty, L^{\frac{p}{2}}(\mathbb{H}^n)} \cr
&\le& \frac{1}{\beta}\left(  e^{-\frac{\beta t}{2}}-e^{-\beta t} \right)   \alpha \tilde C^{\frac{2}{p}} C{k(\gamma)}\norm{\theta}_{\infty, L^{\frac{p}{2}}(\mathbb{H}^n)} \norm{\omega}_{\infty, L^{\frac{p}{2}}(\mathbb{H}^n)} \cr
&\longrightarrow& 0, \hbox{   when   } t \to +\infty.
\end{eqnarray*}
This limit leads to
\begin{equation}\label{lim21}
\lim_{t\to +\infty}\norm{S_3(\theta,\omega)(t)}_{Y}=0.
\end{equation}
On the other hand, we have $\lim\limits_{t\to +\infty}\norm{\theta(t)}_Y = 0$, then for all $\epsilon>0$, there exists $t_0>0$ large enough such that for all $t>t_0$, we have $\norm{\theta(t)}_Y<\epsilon$. {Combining} with the argument from the proof of Lemma \ref{Thm:linear}, this implies
\begin{eqnarray*}
\norm{S_4(\theta,\omega)(t)}_{Y} \leqslant \tilde{K}\alpha {k(\gamma)\norm{\omega}_{\infty, L^{\frac{p}{2}}(\mathbb{H}^n)}}\epsilon\;\; \hbox{   for  } t>t_0,
\end{eqnarray*}
where the constant $\tilde{K}$ is given in Lemma \ref{Thm:linear}. This yields that
\begin{equation}\label{lim22}
\lim_{t\to+\infty}\norm{S_4(\theta,\omega)(t)}_{Y}=0.
\end{equation}
From these limits \eqref{lim21} and \eqref{lim22}, we obtain
\begin{equation}\label{lim2}
\lim_{t\to+\infty}\norm{\tilde{S}(\theta,\omega)(t)}_{Y}=0.
\end{equation}
In other words, $\tilde{S}(\theta,\omega)\in C_0(\r_+,Y)$. 

\underline{\bf Step 3:}
Finally, we prove that the function $t\mapsto e^{t {\Delta_{\mathbb{H}^n}}}\left(u(0) - \hat{S}(H)(0) + \hat{S}(\eta, \eta)(0)\right)$ belongs to $C_0(\r_+,Y)$. Due to Lemma \ref{estimates}, we have
\begin{eqnarray}
&&\left\| e^{t {\Delta_{\mathbb{H}^n}}}\left(u(0)- \hat{S}(H)(0) + \hat{S}(\eta, \eta)(0)\right) \right\|_Y\cr
&&\leq e^{- \tilde{\beta} t}\left( \left\|u(0) \right\|_Y + \norm{\hat{S}(H)(0)}_Y + \norm{\hat{S}(\eta, \eta)(0)}_Y\right)\cr
&&\leq e^{- \tilde{\beta} t}\left( \left\|u(0) \right\|_Y + \tilde{K}\norm{H}_{C_b(\r_+,X)} + \tilde{K} \alpha {k(\gamma)}\norm{\eta}^2_{C_b(\r_+,Y)}\right)  \to 0 \hbox{   when   } t\to +\infty,
\end{eqnarray}
where $ \tilde{\beta} =  \gamma_{\frac{p}{2},\frac{p}{2}}>0$.  Then
\begin{equation*}
\lim_{t\to +\infty}\norm{e^{t {\Delta_{\mathbb{H}^n}}}\left(u(0) - \hat{S}(H)(0) + \hat{S}(\eta,\eta)(0) \right)}_{Y}=0. 
\end{equation*}

Our proof is completed by combining three above steps.
\end{proof}

\section{The semi-linear equations} \label{S4} 
In this section, we consider the Keller-Segel system with initial data on the real hyperbolic space $(\mathbb{H}^n,{\mathfrak g) }$ (where $n\geqslant 2$):
\begin{equation}\label{smKS} 
\left\{
  \begin{array}{rll} 
u_t \!\! & = {\Delta_{\mathbb{H}^n}} u + \nabla_x \cdot \left[ - \alpha u\nabla_x (- {\Delta_{\mathbb{H}^n}}+ \gamma I)^{-1}u + f\right], \; & x\in  \mathbb{H}^n,\, t \in \r_+, \hfill\cr
u(0) \!\!&=\; u_0.
\end{array}\right. 
\end{equation} 
The {mild solution} of the Cauchy problem \eqref{smKS} is given as
\begin{equation}\label{mild:smlinear*}
u(t) = e^{t {\Delta_{\mathbb{H}^n}} }u_0 + \int_0^t \nabla_x \cdot e^{(t-s) {\Delta_{\mathbb{H}^n}}}\left[ -\alpha u\nabla_x(- {\Delta_{\mathbb{H}^n}}+ \gamma I)^{-1}u + f \right](s)ds.
\end{equation}
The aim of this section is to prove the existence, uniqueness of the small asymptotically almost periodic mild solutions. We can apply fixed point arguments and the results from linear equation as above. 
\begin{theorem}\label{thm2.20}
Let $n\geq 2$ and ${\max \left\lbrace 3,n\right\rbrace }<p<2n$. Suppose that a function $f$  belongs  to $AAP(\mathbb{R}_+, L^{\frac{p}{3}}(\mathbb{H}^n))$.  If the norms $\|u_0\|_{\infty,L^{\frac{p}{2}}(\mathbb{H}^n)}$ and  $\|f\|_{\infty,L^{\frac{p}{3}}(\mathbb{H}^n)}$ are sufficiently small,  the equation \eqref{smKS} has one and only one asymptotically almost periodic mild solution $\hat{u}$ on a small ball of  
$C_b(\r, L^{\frac{p}{2}}(\mathbb{H}^n))$.

Moreover, if the function $f$ satisfies that $\sup\limits_{t>0} e^{\sigma t}\norm{f(t)}_{L^{\frac{p}{3}}(\mathbb{H}^n)}<+\infty$, then the above solution $\hat{u}$ is exponential time-decay, i.e., it satisfies
\begin{equation}\label{exponentialdecay}
\norm{\hat{u}(t)}_{L^{\frac{p}{2}}(\mathbb{H}^n)} \leqslant De^{-\sigma t}
\end{equation}
for all $t>0$, where 
$\sigma= \min\left\{ \gamma_{p/2,p/2},\,\frac{\gamma_{p/2,p/2}+\gamma_{pn/(4n-p),p/2}}{2},\, \frac{\gamma_{p/2,p/2}+\gamma_{p/3,p/2}}{2}\right\}$.
\end{theorem} 
\begin{proof}
Let 
\begin{eqnarray}\label{bro}
\B_\rho^{AAP}=\left\{\omega\in AAP(\r_+, L^{\frac{p}{2}}(\mathbb{H}^n)):  \norm{\omega}_{{\infty,L^{\frac{p}{2}}(\mathbb{H}^n)}} \le \rho \right\}
\end{eqnarray}
be the ball centered at zero and radius $\rho$. 
For each $\omega\in \B_\rho^{AAP}$, we consider the linear equation 
\begin{equation}\label{ns1}
u(t) = e^{t {\Delta_{\mathbb{H}^n}} }u_0 + \int_0^t \nabla_x \cdot e^{(t-s) {\Delta_{\mathbb{H}^n}}}\left[ -\alpha \omega\nabla_x(- {\Delta_{\mathbb{H}^n}}+ \gamma I)^{-1}\omega + f \right](s)ds.
\end{equation} 
By Theorem \ref{pest}, Equation \eqref{ns1} has a unique asymptotically almost periodic mild solution $u$ such that
\begin{eqnarray}\label{CoreEstimate}
\norm{u(t)}_{L^{\frac{p}{2}}(\mathbb{H}^n)}& \leq&  \norm{u_0}_{L^{\frac{p}{2}}(\mathbb{H}^n)}  +  \tilde{K}\left( \alpha {k(\gamma)}\norm{\omega}^2_{\infty,L^{\frac{p}{2}}(\mathbb{H}^n)} + \norm{f}_{\infty,L^{\frac{p}{3}}(\mathbb{H}^n)} \right)\cr
& \leq&  \norm{u_0}_{L^{\frac{p}{2}}(\mathbb{H}^n)}  +  \tilde{K}\left( \alpha {k(\gamma)}\rho^2 + \norm{f}_{\infty,L^{\frac{p}{3}}(\mathbb{H}^n)} \right)\cr
&\leq& {\rho}
\end{eqnarray}
if $\norm{u_0}_{L^{\frac{p}{2}}(\mathbb{H}^n)},\, \rho$ and $\norm{f}_{\infty,L^{\frac{p}{3}}(\mathbb{H}^n)}$ are small enough. Therefore, we can define a map as follows
\begin{equation}\label{defphi}
\begin{split}
\Phi:  C_b(\r_+, L^{\frac{p}{2}}(\mathbb{H}^n)) &\to C_b(\r_+, L^{\frac{p}{2}}(\mathbb{H}^n))\\ 
\omega &\mapsto \Phi(\omega)=u
\end{split}
\end{equation}
if $\norm{u_0}_{L^{\frac{p}{2}(\mathbb{H}^n)}}$, $\rho$ and $\norm{f}_{\infty,L^{\frac{p}{3}}(\mathbb{H}^n)}$ are small enough satisfying Estimate \eqref{CoreEstimate}. Hence, the map  $\Phi$ acts from  $\B_\rho^{AAP}$ into itself. It turns out that 
\begin{equation}\label{defphi1}
\Phi(\omega)(t) = e^{t {\Delta_{\mathbb{H}^n}} }u_0 + \int_0^t \nabla_x \cdot e^{(t-s) {\Delta_{\mathbb{H}^n}}}\left[ -\alpha \omega\nabla_x(- {\Delta_{\mathbb{H}^n}}+ \gamma I)^{-1}\omega + f \right](s)ds.
\end{equation}
Moreover,
for $\omega_1, \omega_2\in \B_\rho^{AAP}$, the function $u:=\Phi(\omega_1)-\Phi(\omega_2)$ becomes a unique asymptotically almost periodic mild solution to the equation 
\begin{eqnarray*}
\partial_tu  - \Delta_{\mathbb{H}^n} u   &=& -\alpha \omega_1\nabla_x(- {\Delta_{\mathbb{H}^n}}+ \gamma I)^{-1}\omega_1 +\alpha \omega_2\nabla_x(- {\Delta_{\mathbb{H}^n}}+ \gamma I)^{-1}\omega_2 \cr
&=&{-\alpha \omega_1\nabla_x(- {\Delta_{\mathbb{H}^n}}+ \gamma I)^{-1}(\omega_1-\omega_2) +\alpha (\omega_2-\omega_1)\nabla_x(- {\Delta_{\mathbb{H}^n}}+ \gamma I)^{-1}\omega_2.}
 \end{eqnarray*}
Thus, by \eqref{defphi1} and the same way to establish inequality \eqref{CoreEstimate}, we can estimate
{\begin{eqnarray}\label{Core}
\norm{\Phi(\omega_1)-\Phi(\omega_2)}_{\mathcal{X}} &\leqslant 2\alpha {k(\gamma)} \rho \norm{\omega_1-\omega_2}_{\mathcal{X}}.
\end{eqnarray}  }
Therefore,  $\Phi$ invokes a contradiction if $\rho$ is sufficiently small. This can be shown directly for a unique fixed point $\hat{u}$ of $\Phi$. Due to the definition of $\Phi$, this function $\hat{u}$ is an  asymptotically almost periodic mild solution to Keller-Segel system \eqref{smKS}. The uniqueness of $\hat{u}$ in the small ball $\B_\rho^{AAP}$ is clearly by using \eqref{Core}.

Now we prove the exponential decay of asymptotically almost periodic solution $\hat{u}$. For a given mild solution $\hat{u}$ in the small ball $\cal B_\rho^{AAP}$ of $C_b(\r_+,L^{\frac{p}{2}}(\mathbb{H}^n))$, i.e., $\norm{\hat{u}}_{C_b(\r_+,L^{\frac{p}{2}}(\mathbb{H}^n))}<\rho$, we have
\begin{eqnarray}\label{Gron}
\norm{\hat{u}(t)}_{L^{\frac{p}{2}}(\mathbb{H}^n)} &\leqslant& \norm{e^{t {\Delta_{\mathbb{H}^n}}}u_0}_{L^{\frac{p}{2}}(\mathbb{H}^n)} + \alpha \int_0^t \norm{\nabla_x \cdot e^{(t-s) {\Delta_{\mathbb{H}^n}}}\left[ \hat{u}\nabla_x(- {\Delta_{\mathbb{H}^n}} + \gamma I)^{-1} \hat{u} \right](s)}_{L^{\frac{p}{2}}(\mathbb{H}^n)}ds\cr
&&+\int_0^t \norm{\nabla_x \cdot e^{(t-s) {\Delta_{\mathbb{H}^n}}}f(s)}_{L^{\frac{p}{2}}(\mathbb{H}^n)}ds\cr
&\leq& \hat C e^{-t\gamma_{p/2,p/2}}\norm{u_0}_{L^{\frac{p}{2}}(\mathbb{H}^n)} \cr
&& + \alpha \int_0^t [h_n(t-s)]^{\frac{2}{p}}e^{-(t-s)\beta}\norm{\left[ \hat u\nabla_x(- {\Delta_{\mathbb{H}^n}} + \gamma I)^{-1}\hat u \right](s)}_{L^{\frac{pn}{4n-p}}(\mathbb{H}^n)}ds\cr
&&+ \int_0^t [h_n(t-s)]^{\frac{1}{p}+\frac{1}{n}}e^{-(t-s)\hat\beta}\norm{f(s)}_{L^{\frac{p}{3}}(\mathbb{H}^n)}ds\cr
&\leq& \hat C e^{-t\gamma_{p/2,p/2}}\norm{u_0}_{L^{\frac{p}{2}}(\mathbb{H}^n)} \cr&& + \alpha \int_0^t [h_n(t-s)]^{\frac{2}{p}}e^{-(t-s)\beta}\norm{\hat u(s)}_{L^{\frac{p}{2}}(\mathbb{H}^n)}  \norm{\left[\nabla_x(- {\Delta_{\mathbb{H}^n}} + \gamma I)^{-1}\hat u \right](s)}_{L^{\frac{pn}{2n-p}}(\mathbb{H}^n)}ds\cr
&&+ \int_0^t [h_n(t-s)]^{\frac{1}{p}+\frac{1}{n}}e^{-(t-s)\hat\beta}\norm{f(s)}_{L^{\frac{p}{3}}(\mathbb{H}^n)}ds\cr
&\leq& \hat C e^{-t\gamma_{p/2,p/2}}\norm{u_0}_{L^{\frac{p}{2}}(\mathbb{H}^n)}\cr
&& +\alpha \tilde C^{\frac{2}{p}} C{k(\gamma)} \int_0^t \left( (t-s)^{-\frac{n}{p}} +1 \right) e^{-(t-s)\beta} \norm{\hat u}_{ L^{\frac{p}{2}}(\mathbb{H}^n))}ds \norm{\hat u}_{C_b(\r_+, L^{\frac{p}{2}}(\mathbb{H}^n))}\cr
&&+ \int_0^t [h_n(t-s)]^{\frac{1}{p}+\frac{1}{n}}e^{-(t-s)\hat\beta}\norm{f(s)}_{L^{\frac{p}{3}}(\mathbb{H}^n)}ds\cr
&\leq& \hat C e^{-t\gamma_{p/2,p/2}}\norm{u_0}_{L^{\frac{p}{2}}(\mathbb{H}^n)} +\alpha \tilde C^{\frac{2}{p}} C{k(\gamma)} \rho \int_0^t \left( (t-s)^{-\frac{n}{p}} +1 \right) e^{-(t-s)\beta} \norm{\hat u}_{ L^{\frac{p}{2}}(\mathbb{H}^n))}ds  \cr
&&+ \int_0^t [h_n(t-s)]^{\frac{1}{p}+\frac{1}{n}}e^{-(t-s)\hat\beta}\norm{f(s)}_{L^{\frac{p}{3}}(\mathbb{H}^n)}ds.
\end{eqnarray}
Setting $y(t) = e^{\sigma t}\norm{\hat{u}(t)}_{L^{\frac{p}{2}}(\mathbb{H}^n)}$, where $\sigma= \min\left\{ \gamma_{p/2,p/2},\, \beta,\, \hat \beta \right\}$. The inequality \eqref{Gron} leads to
\begin{eqnarray}\label{Gron1}
y(t) \leq \hat C\norm{u_0}_{L^{\frac{p}{2}}(\mathbb{H}^n)} &+& \int_0^t [h_n(t-s)]^{\frac{1}{p}+\frac{1}{n}}e^{-(t-s)(\hat \beta-\sigma)}ds\sup_{t>0} e^{\sigma t} \norm{f(t)}_{L^{\frac{p}{3}}(\mathbb{H}^n)} \cr
&+&  \alpha \tilde C^{\frac{2}{p}} C {k(\gamma)} \rho  \int_0^t \left( (t-s)^{-\frac{n}{p}} +1 \right) e^{-(t-s)(\beta-\sigma)} y(s)ds.
\end{eqnarray}
Since the boundedness of integrals 
$$\int_0^t [h_n(t-s)]^{\frac{1}{p}+\frac{1}{n}}e^{-(t-s)(\hat \beta-\sigma)}ds<  \frac{\tilde C^{\frac{1}{p}+\frac{1}{n}}}{(\hat \beta-\sigma)^{\frac{1}{2} - \frac{n}{2p}}}{\bf\Gamma}\left( \frac{1}{2} - \frac{n}{2p} \right) + \frac{\tilde C^{\frac{1}{p}+\frac{1}{n}}}{\hat\beta-\sigma}=\hat {D}<+\infty,$$
and 
$$ \int_0^t \left( (t-s)^{-\frac{n}{p}} +1 \right) e^{-(t-s)(\beta-\sigma)}ds<\left( \frac{1}{(\beta-\sigma)^{1- \frac{n}{p}}}\Gamma\left(1- \frac{n}{p} \right) + \frac{1}{\beta-\sigma} \right)=\tilde D<+\infty,   $$
utilizing Gronwall's inequality, we obtain from inequality \eqref{Gron1} that
\begin{equation}
y(t)\leq \left(\hat C \norm{u_0}_{L^{\frac{p}{2}}(\mathbb{H}^n)}+ \hat{D}\sup_{t>0}e^{\sigma t}\norm{f(t)}_{L^{\frac{p}{3}}(\mathbb{H}^n)} \right) e^{\alpha \tilde C^{\frac{2}{p}} C {k(\gamma)} \rho  \tilde{D}}
\end{equation}
for all $t>0$. This yields the exponential decay \eqref{exponentialdecay}.
Our proof is completed.
\end{proof}

\end{document}